\documentclass{colt2015} 

\title[ Learning rates for the dependence structure of rare events]{Learning the dependence structure of rare events: \\ a non-asymptotic study}
\usepackage{times}

 \coltauthor{\Name{Nicolas Goix} \Email{goix@telecom-paristech.fr}\\
  \Name{Anne Sabourin} \Email{sabourin@telecom-paristech.fr}\\
  \Name{St\'ephan Cl\'emen\c con} \Email{clemencon@telecom-paristech.fr}\\
  \addr Institut Mines-T\'el\'ecom, T\'el\'ecom ParisTech, CNRS LTCI}

\usepackage[utf8]{inputenc} 
\usepackage{dsfont}
\usepackage{soul}

\newcommand{\ie}{\emph{i.e.}{}}
\newcommand{\eg}{\emph{e.g.}{}}
\newcommand\iid{\ensuremath{\mathit{i.i.d.}}\ }
\newcommand{\rv}{\emph{r.v.}{}}

\DeclareMathOperator{\argmin}{argmin}

\def\Var{\mathbf{Var}}
\def\P{\mathbb{P}}

\def\mb{\mathbf}
\def\bb{\mathbb}
\def\point{\,\cdot\,}

\usepackage{ulem}

\begin{document} 
\normalem
\maketitle

\begin{abstract} 

Assessing the probability of occurrence of extreme events  is a crucial issue in various fields like finance, insurance, telecommunication or environmental sciences. In a multivariate framework, the tail dependence is characterized by the so-called \emph{stable tail dependence function} (\textsc{stdf}). Learning this structure is the keystone of multivariate extremes. Although extensive studies have proved consistency and asymptotic normality for the empirical version of the \textsc{stdf}, non-asymptotic bounds are still missing. The main purpose of this paper is to fill this gap. Taking advantage of adapted VC-type concentration inequalities, upper bounds are derived with expected rate of convergence in $O(k^{-1/2})$. The concentration tools involved in this analysis rely on a more general study of maximal deviations in low probability regions, and thus directly apply to the classification of extreme data. \\



\end{abstract} 

\section{Introduction}
\label{sec:intro}
Extreme Value Theory (\textsc{EVT}) develops models for learning  the
unusual rather than the usual. These  models are widely used in fields
involving risk management like finance, insurance, telecommunication
or environmental sciences. 
One major application of \textsc{EVT} is to provide a reasonable
assessment of the probability of occurrence of rare events. 
To illustrate this point,  suppose we want to manage the risk of a
  portfolio  containing $d$ different assets,  $\mb
X = (X_1,\ldots,X_d)$.  
A  fairly general purpose is then to evaluate  the
probability of events of the kind 
$\{X_1 \ge x_1 \text{ or }  \dotsc \text{ or }
X_d\ge x_d \}$, for large multivariate thresholds $\mb
x=(x_1,\ldots,x_d)$.  
 Under not too stringent conditions on the regularity of
  $\mb X$'s  distribution, \textsc{EVT} shows that for large enough
  thresholds, (see Section~\ref{sec:background} for details) 
\[
\P\{X_1 \ge x_1 \text{ or }  \dotsc \text{ or }
X_d\ge x_d \} \simeq 
l(p_1,\ldots,p_d), 
\]  
where $l$ is the  \emph{stable tail dependence function} and the
$p_j$'s  are the marginal exceedance probabilities, $p_j = \P(X_j\ge
x_j)$. Thus, the functional $l$  characterizes 
 the \emph{dependence} among extremes. The \emph{joint}   distribution
 (over large thresholds) 
 can thus be recovered from  the knowledge of the marginal distributions  together with
 the \textsc{stdf} $l$. In practice, $l$ can be learned 
 from
 `moderately extreme' data, typically the $k$   `largest' ones among a
 sample of size $n$, with $k\ll n$.
Recovering the $p_j$'s can be done following a well paved way: in the univariate case, \textsc{EVT} essentially consists in modeling
the distribution of the maxima (\emph{resp.} the upper tail) as a generalized
extreme value distribution, namely an element of the Gumbel, Fréchet
or Weibull parametric families (\emph{resp.} by 
 a  generalized Pareto distribution).

In contrast, in the multivariate case, 
there is no finite-dimensional parametrization of the dependence
structure. 
The latter is characterized by 
the so-called \emph{stable tail dependence function} (\textsc{stdf}).
 Estimating this functional is thus one of the main issues in multivariate \textsc{EVT}. Asymptotic properties of the empirical \textsc{stdf} have been widely studied, see \cite{Huangphd}, \cite{Drees98}, \cite{Embrechts2000} and \cite{dHF06} for the bivariate case, and \cite{Qi97}, \cite{Einmahl2012} for the general multivariate case under smoothness assumptions.

However, to the best of our knowledge, no bounds exist  
on the finite sample error. It is precisely the purpose
of this paper to derive such non-asymptotic  bounds. Our results do
not require any assumption other than  the existence of the \textsc{stdf}.
The main idea is as follows. The empirical estimator is based on 
the empirical measure of `extreme' regions, which  are hit
 only with  low probability. It is thus enough to bound 
 maximal deviations on such low probability regions. The key consists
 in choosing an adaptive VC class, which only covers the
 latter regions, and on the other hand, to  derive  VC-type inequalities that incorporate $p$,
 the probability of hitting the class at all.

The structure of the paper is as follows. The whys and wherefores of
 \textsc{EVT} and the \textsc{stdf} are explained in
 Section~\ref{sec:background}. In Section \ref{sec:concentration},
 concentration tools which rely on the general study of maximal
 deviations in low probability regions are introduced, with an
 immediate application to the framework of classification (Remark
 \ref{rk:prediction}). The main result of the paper,  a
 non-asymptotic bound on the convergence of the empirical
 \textsc{stdf}, is derived in Section \ref{sec:stdf}. Section~\ref{sec:conclusion} concludes.

\section{Background in extreme value theory}\label{sec:background}
A  useful setting to understand the use  of \textsc{EVT} and to give
intuition about the \textsc{stdf} concept is that  of risk monitoring. 
In the univariate case, it is natural to consider the $(1-p)^{th}$
quantile of the distribution  $F$ of a random
variable $X$,  for a given exceedance probability $p$, that is
$x_p = \inf\{x \in \mathbb{R},~ \mathbb{P}(X > x) \le p\}$. For
moderate values of $p$, a natural empirical estimate is  $x_{p,n} = \inf\{x \in
\mathbb{R},~ 1/n \sum_{i=1}^n \mathds{1}_{X_i > x}\le p\}$.
However,  if
$p$ is very small
, the finite  sample $X_1,
\ldots, X_n$  contains insufficient information and $x_{p,n}$ becomes 
irrelevant. 
That is where \textsc{EVT} comes into play  by providing
parametric estimates of large
quantiles: 
whereas statistical inference often involves sample means and the
central limit
theorem, 
\textsc{EVT} handles phenomena whose behavior is 
not ruled by an `averaging effect'. The focus is on the sample maximum
rather than the mean. The primal assumption is the existence of two
sequences $\{a_n, n \ge 1\}$ and $\{b_n, n \ge 1\}$, the $a_n$'s being
positive, and a non-degenerate distribution function $G$ such that
\begin{align}
\label{intro:assumption1}
\lim_{n \to \infty} n ~\mathbb{P}\left( \frac{X - b_n}{a_n} ~\ge~ x \right) = -\log G(x)
\end{align}
\noindent
for all continuity points $x \in \mathbb{R}$ of $G$.
If this assumption is fulfilled -- it is the case for most  textbook
distributions -- then $F$ is said to be in the \textit{domain of
  attraction} of $G$, denoted $F \in DA(G)$. The tail behavior of $F$
is then essentially characterized by $G$, which is proved to be -- up
to  rescaling -- of the type $G(x) = \exp(-(1 + \gamma
x)^{-1/\gamma})$ for $1 + \gamma x > 0$, $\gamma \in \mathbb{R}$,
setting by convention $(1 + \gamma x)^{-1/\gamma} = e^{-x}$ for
$\gamma = 0$. The sign of $\gamma$ controls the shape of the tail and
various estimators of the rescaling sequence as well as $\gamma$ have
been studied in great detail, see \emph{e.g.}  \cite{DEd1989},
\cite{ELL2009}, 
\cite{Hill1975}, \cite{Smith1987}, \cite{BVT1996}. 

In the multivariate case, it is mathematically very convenient to decompose the joint
distribution of $\mb X = (X^1,\ldots, X^d)$  
into the margins on the
one hand, and the dependence structure on the other hand. In
particular, handling uniform margins is very helpful when it comes
to establishing upper bounds on the deviations between empirical and mean
 measures. Define thus  standardized variables $U^j = 1 - F_j(X^j)$,
 where $F_j$ is the marginal distribution function of $X^j$, and
 $\mathbf{U} = (U^1,\dotsc,U^d)$. Knowledge of the $F_j$'s and of the
 joint distribution of $\mb U$ allows to recover that of $\mb X$, 
 since  $\P(X_1\le x_1,\ldots, X_d\le x_d) = \P(U^1 \ge
 1 -  F_1(x_1),\ldots,U^d\ge 1-F_d(x_d))$.  
  With these notations,  under a fairly general assumption
similar to 
(\ref{intro:assumption1}) 
(namely, standard multivariate regular variation of standardized variables, see \eg~\cite{Resnick2007}, chap. 6), there exists a limit measure
$\Lambda$ on $[0,\infty]^d \setminus\{\infty\}$ (called the
\emph{exponent measure}) such that 
\begin{align}
\label{stdf1}
\lim_{t \to 0} t^{-1} \mathbb{P} \left[
 U^1 \le t\, x_1
   ~\text{or}~ \ldots ~\text{or}~ U^d \le t\,x_d  \right]
 = \Lambda[\mb x, \infty]^c :=l(\mb x)~. \qquad(x_j\in[0,\infty], \mb x\neq\infty)
\end{align}
Notice that no assumption is made about the marginal distributions, so that our framework allows non-standard regular variation, or even no regular variation at all of the original data $\mb X$ (for more details see \eg~\cite{Resnick2007}, th. 6.5 or \cite{Resnick1987}, prop. 5.10.).
The functional $l$ in the limit in (\ref{stdf1}) is called the \emph{stable tail
  dependence function}. 
In the remainder of this paper, the only assumption is the existence
of a limit in (\ref{stdf1}), \ie, the existence of the \textsc{stdf}.

We emphasize that the knowledge of both $l$ and the margins gives
access to the probability of hitting `extreme' regions of the kind
$[\mb 0, \mb x]^c$, for  `large' thresholds 
$\mb x = (x_1,\ldots,x_d)$ (\ie~such that for some $ j \le d$, $ 1-F_j(x_j) $ is   a
$O(t)$   for some small
$t$). Indeed, in such a case, 
\begin{align*}
\mathbb{P}(X^1>x_1 \text{ or } \ldots \text{ or } X^d>x_d) 
&= \mathbb{P}\left(\bigcup_{j=1}^d (1-F_j)(X^j) \le (1-F_j)(x_j)
\right)  \\
&= t \, \left\{\frac{1}{t}\mathbb{P}\left(\bigcup_{j=1}^d U^j \le t\,
    \left[\frac{ (1-F_j)(x_j)}{t}  \right]
\right)\right\} \\
&\underset{t\to 0}{\sim} \; ~t ~l\Big( t^{-1}\,(1-F_1)(x_1),\;  \ldots, \;  t^{-1}\, (1-F_d)(x_d) \Big)  \\
&=  ~~l\Big((1-F_1)(x_1), \;\ldots,\; (1-F_d)(x_d) \Big)     
\end{align*}
where the  last equality follows from the homogeneity of $l$.
  This underlines the utmost importance of estimating the \textsc{stdf} and
by extension stating non-asymptotic bounds on this convergence.

\noindent
Any stable tail dependence function $l(.)$ is in fact a norm, (see \cite{Falk94}, p179) and satisfies $$\max\{x_1,\ldots, x_n\} ~\le~ l(\mathbf{x}) ~\le~ x_1 + \ldots + x_d, $$ where the lower bound is attained if $\mathbf{X}$ is perfectly tail dependent (extremes of univariate marginals always occur simultaneously), and the upper bound in case of tail independence or asymptotic independence (extremes of univariate marginals never occur simultaneously).
We refer to \cite{Falk94} for more details and properties on the
\textsc{stdf}.

\section{A VC-type inequality adapted to the study of low probability regions}
\label{sec:concentration}

%
%
%
Classical VC inequalities aim at bounding the deviation of empirical
from theoretical quantities on relatively simple classes of sets,
called VC classes. These classes typically cover the support of the
underlying distribution.  However, when dealing with rare
events, 
it is of great interest to have such bounds on a class of sets
which only covers a small probability region and thus contains (very)
few 
observations. This yields sharper
bounds,  
 since only differences  between very small quantities are
involved. 
The starting point of this analysis is the following VC-inequality stated below.




\begin{theorem}
\label{thm-princ} 
Let $\mathbf{X}_1,\ldots,\mathbf{X}_n$ \iid~realizations of a \rv~$\mathbf{X}$, a VC-class $\mathcal{A}$ with VC-dimension $V_{\mathcal{A}}$ and shattering coefficient (or growth function) $S_{\mathcal{A}}(n)$.
Consider the class union $\mathbb{A} = \cup_{A \in \mathcal{A}} A$,
 and let  
$p = \mathbb{P}(\mathbf{X} \in \mathbb{A})$. Then there is an absolute constant $C$ such that for all $0<\delta<1$, with probability at least $1-\delta$,
\begin{align}
\label{thm-princ-ineq}
\sup_{A \in \mathcal{A}} \left| \mathbb{P} \big[\mathbf{X} \in A\big] - \frac{1}{n} \sum_{i=1}^n \mathds{1}_{\mathbf{X}_i \in A}  \right| ~~\le~~ C \bigg[ \sqrt{p}\sqrt{\frac{V_{\mathcal{A}}}{n} \log{\frac{1}{\delta}}} + \frac{1}{n} \log{\frac{1}{\delta}} \bigg]~.
\end{align}
\end{theorem}

\begin{proof} (sketch of)
Details of the proof are deferred to the appendix section.
We use a Bernstein-type concentration inequality (\cite{McDiarmid98}) that we  apply to 
the general  functional \[f(\mathbf{X}_{1:n})= \sup_{A \in \mathcal{A}} \left |
  \mathbb{P}(\mathbf{X} \in A) - \frac{1}{n} \sum_{i=1}^n
  \mathds{1}_{\mathbf{X}_i \in A} \right|~,\]
where $\mb X_{1:n}$ denotes the sample $(\mb X_1,\ldots,\mb X_n)$.
The inequality in \cite{McDiarmid98} 
involves
the variance of the
\rv~$f(\mathbf{X}_1,\ldots,\mathbf{X}_{k}, x_{k+1},\ldots, x_n) -
f(\mathbf{X}_1,\ldots,\mathbf{X}_{k-1},x_k,\ldots, x_n)$, which can
easily be bounded in our setting. 
We obtain 
\begin{align}
\label{thm-princ:general}
\mathbb{P}\left [ f(\mb X_{1:n}) - \mathbb{E} f(\mb X_{1:n}) ~\ge~ t \right] ~\le~ e^{-\frac{n t^2}{2q + \frac{2t}{3}} },
\end{align}

\noindent
where the quantity $q~=~ \mathbb{E}\left ( \sup_{A \in \mathcal{A}}
  \left | \mathds{1}_{\mathbf{X}' \in A} - \mathds{1}_{\mathbf{X} \in
      A} \right|\right)$ (with  $\mathbf{X}'$ an independent copy of $\mathbf{X}$) is  a measure of the complexity of the class $\mathcal{A}$ with respect to the distribution of $\mathbf{X}$. 
It leads to high probability bounds on $f(\mb X_{1:n})$ of the form
$\mathbb{E}f(\mb X_{1:n}) + \frac{1}{n} \log (1/\delta) +
\sqrt{\frac{2q}{n} \log (1/\delta)} $ instead of the standard
Hoeffding-type bound  $~\mathbb{E}f(\mb X_{1:n}) + \sqrt{\frac{1}{n} \log (1/\delta)}$ .
It is then easy to see that $q \le 2\sup_{A \in \mathcal{A}} \mathbb{P}(\mathbf{X} \in A) \le 2p.$
Finally, an  upper bound on
$\bb E f(\mb X_{1:n})$ is obtained  by introducing re-normalized Rademacher averages 
\begin{align*}
\mathcal{R}_{n,p} = \mathbb{E} \sup_{A \in \mathcal{A}} \frac{1}{np} \left | \sum_{i=1}^{n} \sigma_i \mathds{1}_{\mathbf{\mathbf{X}}_i \in A}\right|~. 
\end{align*}
which are then proved to be of order $O (\sqrt{\frac{V_\mathcal{A} }{pn}})$, so that
$\mathbb{E}(f(\mb X_{1:n})) \le C\sqrt{\frac{V_\mathcal{A} }{pn}}.$
\end{proof}


\begin{remark} (\textsc{Comparison with Existing Bounds})
The following re-normalized VC-inequality due to Vapnik and Chervonenkis (see \cite{Vapnik74}, \cite{Anthony93} or \cite{Bousquet04}, Thm 7),
\begin{align} \label{normalize-vc}
\sup_{A \in \mathcal{A}} \left| \frac{ \mathbb{P} (\mathbf{X} \in A) - \frac{1}{n} \sum_{i=1}^n \mathds{1}_{\mathbf{X}_i \in A}  }{\sqrt{\mathbb{P}(\mathbf{X} \in A)}} \right| ~~\le~~ 2 \sqrt{\frac{\log{S_{\mathcal{A}}(2n)}+\log{\frac{4}{\delta}}}{n}}~,
\end{align}
which holds under the same conditions as Theorem~\ref{thm-princ}, allows to derive a bound similar to (\ref{thm-princ-ineq}), but with an additional $\log n$ factor.
Indeed, it is known as Sauer's Lemma (see \cite{Bousquet04}-lemma 1 for instance) that for $n \ge V_{\mathcal{A}}$, $S_{\mathcal{A}}(n) \le (\frac{en}{V_{\mathcal{A}}})^{V_{\mathcal{A}}}$. It is then easy to see from (\ref{normalize-vc}) that:
\begin{align*}
\sup_{A \in \mathcal{A}} \left | \mathbb{P} (\mathbf{X} \in A) - \frac{1}{n} \sum_{i=1}^n \mathds{1}_{\mathbf{X}_i \in A}  \right| ~~\le~~ 2 \sqrt{\sup_{A \in \mathcal{A}}\mathbb{P}(\mathbf{X} \in A)}  \sqrt{\frac{V_{\mathcal{A}}\log{\frac{2en}{V_{\mathcal{A}}}}+\log{\frac{4}{\delta}}}{n}}~.
\end{align*}
Introduce the union $\mathbb{A}$ of all sets  in the  considered VC
class, $\mathbb{A} = \cup_{A \in \mathcal{A}} A$, and let $p =
\mathbb{P}\left ( \mathbf{X} \in \mathbb{A}\right)$. Then, the  previous bound immediately yields 
\begin{align*}
\sup_{A \in \mathcal{A}} \left | \mathbb{P} (\mathbf{X} \in A) - \frac{1}{n} \sum_{i=1}^n \mathds{1}_{\mathbf{X}_i \in A}  \right| ~~\le~~ 2\sqrt p  \sqrt{\frac{V_{\mathcal{A}}\log{\frac{2en}{V_{\mathcal{A}}}}+\log{\frac{4}{\delta}}}{n}}~.
\end{align*}

\end{remark}

\noindent

\begin{remark} (\textsc{Simpler Bound})
If we assume furthermore that $\delta \ge e^{-np}$, then we have:
\begin{align*}
\sup_{A \in \mathcal{A}} \left | \mathbb{P}(\mathbf{X} \in A) - \frac{1}{n} \sum_{i=1}^n \mathds{1}_{\mathbf{X}_i \in A} \right| ~~\le~~ C \sqrt{p} \sqrt{\frac{V_{\mathcal{A}}}{n} \log{\frac{1}{\delta}}}~.
\end{align*}
\end{remark}

\begin{remark} (\textsc{Interpretation})
\label{rk:interpretation}
Inequality (\ref{thm-princ-ineq}) can be seen as an interpolation
between the best case (small $p$) where the rate of convergence is
$O(1/n)$,  
and the worst case (large $p$) where the rate is $O(1/\sqrt{n})$.
An alternative interpretation is as follows: divide both sides of
(\ref{thm-princ-ineq}) by $p$, so that the left hand side becomes  a
supremum of 
conditional probabilities   upon belonging to the union class
$\mathbb{A}$, $\{\mathbb{P}(\mathbf{X}\in A \big|\mathbf{X}\in \mathbb{A}) \}_{A\in\mathbb{A}}$. Then the upper bound is proportional to $\epsilon(np, \delta)$ where $\epsilon(n, \delta) :=\sqrt{\frac{V_{\mathcal{A}}}{n} \log{\frac{1}{\delta}}} + \frac{1}{n} \log{\frac{1}{\delta}}$ is a classical VC-bound; $np$ is in fact the expected number of observations involved in (\ref{thm-princ-ineq}), and can thus be viewed as the effective sample size. 
\end{remark}


\begin{remark} \label{rk:prediction}(\textsc{Classification of Extremes})
A key issue in the prediction framework is to find upper bounds for the maximal deviation $\sup_{g \in \mathcal{G}}|L_n(g) - L(g)|$, where $L(g) = \mathbb{P}(g(\mathbf{X}) \neq Y)$ is the risk of the classifier $g: \mathcal{X} \to \{-1, 1\}$, associated with the \rv~$(\mathbf{X},Y) \in \mathbb{R}^d \times \{-1,1\}$. $L_n(g) = \frac{1}{n} \sum_{i=1}^n \mathds{I}\{g(\mathbf{X}_i)\neq Y_i\} $ is the empirical risk based on a training dataset $\{(\mathbf{X}_1,Y_1),\; \ldots,\; (\mathbf{X}_n,Y_n)  \}$. Strong upper bounds on $\sup_{g \in \mathcal{G}}|L_n(g) - L(g)|$ ensure the accuracy of the empirical risk minimizer $g_n:= \argmin_{g \in \mathcal{G}}L_n(g)$. 

In a wide variety of applications (\textit{e.g.} Finance, Insurance, Networks), it is of crucial importance to predict the system response $Y$ when the input variable $\mathbf{X}$ takes extreme values, corresponding to shocks on the underlying mechanism. In such a case, the risk of a prediction rule $g(\mathbf{X})$ should be defined by integrating the loss function $L(g)$ with respect to the conditional joint distribution of the pair $(\mathbf{X},Y)$ given $\mathbf{X}$ is extreme. For instance, consider the event $\{\|\mathbf{X}\| \ge t_\alpha\}$ where $t_\alpha$ is the $(1-\alpha)^{th}$ quantile of $\|\mathbf{X}\|$ for a small $\alpha$. To investigate the accuracy of a classifier $g$ given $\{\|\mathbf{X}\| \ge t_\alpha\}$,
introduce 
\begin{align*}
L_{\alpha}(g):~=~ \frac{1}{\alpha}\mathbb{P}\left(Y\neq g(\mathbf{X}),~ \| \mathbf{X}\|>t_\alpha \right)~=~\mathbb{P}\left(Y \neq g(\mathbf{X}) ~\big|~ \|\mathbf{X}\| \ge t_\alpha \right)~,
\end{align*}
\noindent
 and its empirical
 counterpart \[L_{\alpha,n}(g):~=~\frac{1}{n\alpha}\sum_{i=1}^n\mathds{I}_{\{Y_i\neq
   g(\mathbf{X}_i),~ \| \mathbf{X}_i\| > \|  \mathbf{X}_{(\lfloor n\alpha \rfloor)} \|  \}}~,\]
 where $\| \mathbf{X}_{(1)}\| \geq \ldots \geq \| \mathbf{X}_{(n)}\|$ are the order
 statistics of $\| \mathbf{X}\|$. Then as an application of Theorem \ref{thm-princ} with $\mathcal{A} = \{(\mathbf{x},y), g(\mathbf{x}) \neq y, \|\mathbf{x}\| > t_\alpha\},~g\in\mathcal{G},$ we have : 
\begin{align}
\label{prediction:rates}
\sup_{g\in \mathcal{G}} \bigg| \widehat{L}_{\alpha, n}(g)- L_{\alpha}(g) \bigg|  \le C \bigg[ \sqrt{\frac{V_{\mathcal{G}}}{n\alpha} \log \frac{1}{\delta}} + \frac{1}{n\alpha} \log{\frac{1}{\delta}} \bigg]~.
\end{align}
We refer to the appendix for more details. Again the obtained rate by
empirical risk minimization  meets our expectations (see remark \ref{rk:interpretation}), insofar as $\alpha$ is the fraction 
of the dataset involved in the empirical risk $L_{\alpha, n}$. We point out that $\alpha$ may typically depend on $n$, $\alpha = \alpha_n \to 0$.
In this context a direct use of the 
standard version of the \textsc{VC} inequality would lead to a rate  of order $1/(\alpha_n\sqrt{n})$, which may not vanish as $n\rightarrow +\infty$ and even go to infinity if $\alpha_n$ decays to $0$ faster than $1/\sqrt{n}$ . 

Let us point out that rare events may be chosen more general than
$\{\|\mathbf{X}\| > t_\alpha \}$, say $\{\mathbf{X} \in Q \}$ with unknown probability
$q=\mathbb{P}(\{\mathbf{X} \in Q \})$. The previous result still applies with
$\widetilde L_Q(g) := \mathbb{P}\left ( Y \neq g(\mathbf{X}), \mathbf{X} \in Q\right)$
and $\widetilde L_{Q,n}(g) := \mathbb{P}_n\left ( Y \neq g(\mathbf{X}), \mathbf{X} \in
  Q\right)$; then the obtained upper bound on $\sup_{g \in
  \mathcal{G}} \frac{1}{q} \left |  \widetilde L_Q(g) - \widetilde
  L_{Q,n}(g) \right|$ is of order $O(1/\sqrt{qn}). $ 

Similar results can be established for the problem of \textit{distribution-free regression}, when the error of any predictive rule $f(\mathbf{x})$ is measured by the conditional mean squared error $\mathbb{E}[(Z-f(\mathbf{X}))^2\mid Z>q_{\alpha_n}]$, denoting by $Z$ the real-valued output variable to be predicted from $\mathbf{X}$ and by $q_{\alpha}$ its quantile at level $1-\alpha$.

\end{remark}

\section{A bound on the STDF}
\label{sec:stdf}

Let us place ourselves in the multivariate extreme framework introduced in Section \ref{sec:intro}:
 Consider a random variable $\mathbf{X} = (X^1, \ldots X^d)$ in
 $\mathbb{R}^d$ with 
distribution function $F$ and marginal distribution functions
$F_1,\ldots,F_d$.
Let $\mathbf{X_1,X_2,\ldots,X_n}$ be an \iid~sample distributed as $\mb X$.
In the subsequent analysis, the only assumption is the existence of
the \textsc{stdf} defined in  (\ref{stdf1}) and
the margins $F_j$ are supposed to be unknown. The
definition of $l$ may be  recast as
\begin{align}
\label{stdf}
l(\mathbf{x}):= \lim_{t \to 0} t^{-1} \tilde F (t\mathbf{x}) 
\end{align}
\noindent
with $\tilde F (\mathbf{x}) = (1-F) \big( (1-F_1)^\leftarrow(x_1),\ldots,
(1-F_d)^\leftarrow(x_d)  \big)$. Here the notation
$(1-F_j)^\leftarrow(x_j)$ denotes the quantity $\sup\{y\,:\; 1-F_j(y)
\ge x_j\}$. Notice that, in terms of standardized variables $U^j$, 
$\tilde F(\mb x) = \P\Big(\bigcup_{j=1}^d\{U^j\le x_j\}\Big) = \P(\mb
U\in [\mb x, \infty[^c)$.

Let $k=k(n)$ be a sequence of positive integers such that $k \to
\infty$ and $k=o(n)$ as $n \to \infty$.
A natural estimator of $l$ is its empirical version defined as
follows,  see \cite{Huangphd}, \cite{Qi97}, \cite{Drees98}, \cite{Einmahl2006}:
\begin{align}
\label{ln}
l_n(\mathbf{x})=\frac{1}{k}~\sum_{i=1}^{n} \mathds{1}_{\{X_i^1 \ge X^1_{(n-\lfloor kx_1 \rfloor+1)} \text{~~or~~} \ldots \text{~~or~~} X_i^d \ge X^d_{(n-\lfloor kx_d\rfloor+1)} \}}~,
\end{align}
\noindent
 The expression is indeed suggested by the definition of $l$ in
 (\ref{stdf}), with all distribution functions and  univariate
 quantiles replaced by their empirical counterparts, and with $t$
 replaced by $k/n$. Extensive studies have proved consistency and 
 asymptotic normality of this nonparametric estimator of $l$, see \cite{Huangphd}, \cite{Drees98} and \cite{dHF06} for the asymptotic normality in dimension $2$, \cite{Qi97} for consistency in arbitrary dimension, and \cite{Einmahl2012} for asymptotic normality in arbitrary dimension under differentiability conditions on $l$.

To our best knowledge, there is no established non-asymptotic bound on the maximal deviation $\sup_{0 \le \mathbf{x} \le T} \left| l_n(\mathbf{x}) - l(\mathbf{x}) \right|$. It is the purpose of the remainder of this section to derive such a bound, without any smoothness condition on $l$.

First, Theorem \ref{thm-princ} needs adaptation  to a   particular
setting: introduce  a random vector 
$\mathbf{Z}=(Z^1,\ldots,Z^d)$ with uniform margins, \ie, for every
$j=1,\ldots,d$, the variable $Z^j$ is uniform on $[0,1]$.  Consider
the class 
\[
\mathcal{A} = \left\{ \Big[  \frac{k}{n}\,
  \mathbf{x},\infty\Big[^{~c} \;:\quad \mb x \in \mathbb{R}^d_+ , \quad 0 \le x_j
\le T \; (1\le j\le d) \right\}
\]
This is a VC-class of
VC-dimension $d$,  as proved in \cite{Devroye96}, Theorem 13.8, for
its complementary class  $\big\{[\mathbf{x},\infty[ ,~ \mathbf{x}>0
\big\}$. 
In this context, the union class $\mathbb{A}$ has mass $p \le dT\frac{k}{n}$ since 
\begin{align*}
\mathbb{P}(\mathbf{Z} \in \mathbb{A}) = \mathbb{P} \left[ \mathbf{Z} \in \left(\Big[\frac{k}{n}T,\infty\Big[^d\right)^c\right] = \mathbb{P} \left[ \bigcup_{j=1..d} \mathbf{Z}^j < \frac{k}{n}T \right] \le \sum_{j=1}^d \mathbb{P} \left[ \mathbf{Z}^j < \frac{k}{n}T \right]
\end{align*}
\noindent
Consider the measures $C_n(\point)=\frac{1}{n} \sum_{i=1}^{n}
\mathds{1}_{\{Z_i \in \point \}}$ and $C(\mathbf{x})=\mathbb{P}(Z \in
\point)$. As a direct consequence of Theorem \ref{thm-princ} the
following inequality holds true  with probability at least $1-\delta$,

\begin{align*}
\sup_{0 \le \mathbf{x} \le T} \frac{n}{k} \left | C_n(\frac{k}{n} [\mathbf{x},\infty[^c) - C(\frac{k}{n} [\mathbf{x},\infty[^c)  \right| ~\le~ C d\left(\sqrt{\frac{T}{k} \log{\frac{1}{\delta}}} ~+~ \frac{1}{k} \log \frac{1}{\delta} \right)~.
\end{align*}
If we assume furthermore that $\delta \ge e^{-k}$, then we have 
\begin{align}
\label{Qialt2}
\sup_{0 \le \mathbf{x} \le T} \frac{n}{k} \left | C_n(\frac{k}{n} [\mathbf{x},\infty[^c) - C(\frac{k}{n} [\mathbf{x},\infty[^c)  \right| ~\le~ C d\sqrt{\frac{T}{k} \log{\frac{1}{\delta}}}~.
\end{align}
\noindent
Inequality (\ref{Qialt2}) is the cornerstone of the following theorem, which is the main result of the paper.
In the sequel, we consider a sequence $k(n)$ of integers such that $k=
o(n)$ and $k(n) \to \infty$. For notational convenience, we often 
drop the dependence in $n$ and simply write $k$ instead of $k(n)$. 
\begin{theorem}
\label{thm:l}
Let $T$ be a positive number such that $T \ge \frac{7}{2}(\frac{\log d}{k} + 1)$, and $\delta$ such that $\delta \ge e^{-k}$. Then there is an absolute constant $C$ such that for each $n >0$, with probability at least $1-\delta$:
\begin{align}
\label{thm:l:ineq}
\sup_{0 \le \mathbf{x} \le T} \left| l_n(\mathbf{x}) - l(\mathbf{x})
\right| ~\le~ Cd\sqrt{\frac{T}{k}\log\frac{d+3}{\delta}} ~+~ \sup_{0
  \le \mathbf{x} \le 2T}\left|\frac{n}{k} \tilde
  F(\frac{k}{n}\mathbf{x})- 
l(\mb x)\right|
\end{align}
\end{theorem}
The second term on the right hand side of (\ref{thm:l:ineq}) is
a  bias term which depends on 
the  discrepancy between the left hand side and the limit in
  (\ref{stdf1}) or (\ref{stdf}) at level $t=k/n$. 
The value $k$ can be interpreted as the effective number of observations  used in the empirical estimate, \ie~the effective sample size for tail estimation. 
Considering classical inequalities in empirical process theory such as
VC-bounds, it is thus no surprise to obtain one  in $O(1/\sqrt k)$.
Too large values of $k$ tend to yield a large bias, whereas too small values of $k$ yield a large variance. For a more detailed discussion on the choice of $k$ we recommend \cite{ELL2009}.

The proof of Theorem~\ref{thm:l} follows the same lines as in \cite{Qi97}.
For  unidimensional random variables $Y_1,\ldots,Y_n$, let us denote
by $Y_{(1)} \le \ldots\le Y_{(n)}$ their order statistics. Define 
then the empirical version $\tilde F_n$ of $\tilde F$ ( introduced in
(\ref{stdf})) as 
\begin{align*}
 \tilde F_n(\mathbf{x})  ~=~ \frac{1}{n} \sum_{i=1}^n \mathds{1}_{\{ U_i^1 \le x_1 ~\text{or}~\ldots~\text{or}~ U_i^d \le x_d \}}~ ,
\end{align*}
so that 
$
  \frac{n}{k} \tilde F_n(\frac{k}{n}\mathbf{x}) ~=~ \frac{1}{k}
  \sum_{i=1}^n \mathds{1}_{\{ U_i^1 \le \frac{k}{n}x_1 ~\text{or}~\ldots~\text{or}~
    U_i^d \le \frac{k}{n}x_d
    \}}~
$. 
\noindent Notice that the $U_i^j$'s are  not observable (since $F_j$ is
unknown). In fact, $\tilde F_n$ will be used as a substitute for $l_n$
 allowing to handle uniform variables. The 
 following lemmas make this point explicit. 

\begin{lemma}[Link between $l_n$ and $\tilde F_n$]
\label{ln-Fn}
The  empirical version of $\tilde F$ and that of $l$ are related \emph{via}
\begin{align*}
l_n(\mathbf{x})~=~\frac{n}{k} \tilde F_n(U_{(\lfloor kx_1\rfloor)}^1,~\ldots~, U_{(\lfloor kx_d \rfloor)}^d).
\end{align*}
\end{lemma}

\begin{proof}
Consider the definition of $l_n$ in (\ref{ln}), and note that for $j=1,\ldots,d$, 
\begin{align*}
 X_i^j \ge X_{(n-\lfloor  kx_i \rfloor +1)}^j &~\Leftrightarrow~ rank(X_i^j) \ge n-\lfloor  kx_j \rfloor+1 \\ &~\Leftrightarrow~  rank( F_j(X_i^j)) \ge n-\lfloor kx_j\rfloor+1 \\ &~\Leftrightarrow~  rank(1-F_j(X_i^j)) \le \lfloor kx_j\rfloor\\ &~\Leftrightarrow~  U_i^j \le U_{(\lfloor kx_j\rfloor)}^j,
\end{align*}
 so that 
$l_n(\mathbf{x})~=~\frac{1}{k}~\sum_{j=1}^n$ $\mathds{1}_{\{ U_j^1 \le U_{(\lfloor kx_1\rfloor)}^1 ~\text{or}~\ldots~\text{or}~ U_j^d \le U_{(\lfloor kx_d\rfloor)}^d  \}}$.
\end{proof}
~\\

\begin{lemma}[Uniform bound on $\tilde F_n$'s deviations]
\label{Fn-tildeF}
 For any finite  $T>0$, and $\delta\ge e^{-k}$,  with probability at least
$1-\delta$, the  deviation of $\tilde F_n$
from  $\tilde F$ is uniformly bounded: 
\begin{align*}
\sup_{0 \le \mathbf{x} \le T}  \left| \frac{n}{k} \tilde F_n(\frac{k}{n}\mathbf{x})-\frac{n}{k} \tilde F ( \frac{k}{n} \mathbf{x}) \right| \le Cd\sqrt{\frac{T}{k}\log{\frac{1}{\delta}}}
\end{align*}

\end{lemma}
\begin{proof}
Notice that 
\[\sup_{0 \le \mathbf{x} \le T} \left| 
  \frac{n}{k} \tilde F_n(\frac{k}{n}\mathbf{x})- \frac{n}{k} \tilde F
  ( \frac{k}{n} \mathbf{x}) \right| = 
\frac{n}{k} \left|
 \frac{1}{n}  \sum_{i=1}^n \mathds{1}_{\{ \mb U_i \in \frac{k}{n}
   ]\mathbf{x},\infty]^c \}} -
   \mathbb{P} \left [\mathbf{U} \in \frac{k}{n} ]\mathbf{x},\infty]^c
   \right] \right|, \] and apply
inequality (\ref{Qialt2}).
\end{proof}

\begin{lemma}[Bound on the order statistics of $\mb U$]
\label{U-x} 
Let $\delta\ge e^{-k}$. For any finite positive number $T>0$ such that $T \ge 7/2((\log d)/k + 1)$, we have with probability greater than $1 - \delta$, 
\begin{align}
\label{eq-Wellner}
\forall~ 1\le j \le d,~~~~~\frac{n}{k} U_{(\lfloor kT\rfloor )}^j ~\le~ 2T~,
\end{align}
and with probability greater than $1- (d+1)\delta$, 
\begin{align*}
\max_{1 \le j \le d}~ \sup_{0 \le x_j \le T} \left| \frac{\lfloor kx_j\rfloor }{k} - \frac{n}{k} U_{(\lfloor kx_j\rfloor )}^j  \right| ~\le~ C\sqrt{\frac{T}{k}\log{\frac{1}{\delta}}}~.
\end{align*}
\end{lemma}

\begin{proof}
  Notice that $\sup_{[0 , T]} \frac{n}{k} U_{(\lfloor k\point\rfloor )}^j =
  \frac{n}{k} U_{(\lfloor kT\rfloor )}^j $ and let $\Gamma_n(t) = \frac{1}{n}
  \sum_{i=1}^n \mathds{1}_{\{U_i^j \le t\}}$ . It then straightforward to see
  that 
\[ \frac{n}{k} U_{(\lfloor kT\rfloor )}^j \le 2T ~~\Leftrightarrow~~
  \Gamma_n\Big(\frac{k}{n} 2T\Big) \ge \frac{\lfloor kT\rfloor }{n} \]
 so that 
\[\mathbb{P}
  \left( \frac{n}{k} U_{(\lfloor kT\rfloor )}^j > 2T \right) ~\le~ \mathbb{P} \left
    ( \sup_{\frac{2kT}{n} \le t \le 1} \frac{t}{\Gamma_n(t)} > 2
  \right). \]
 Using \cite{Wellner78}, Lemma 1-(ii) (we use the fact that,   with
  the notations of this reference, $h(1/2) \ge 1/7$ ), we obtain 
\[\mathbb{P} \left( \frac{n}{k} U_{(\lfloor kT\rfloor )}^j > 2T \right) \le
  e^{-\frac{2kT}{7}}, \]
 and thus
 $$\mathbb{P} \left( \exists j,~
    \frac{n}{k} U_{(\lfloor kT\rfloor )}^j > 2T \right) \le de^{-\frac{2kT}{7}} \le
  e^{-k} \le \delta  
  $$ as required in (\ref{eq-Wellner}).  Yet,
\begin{align*}
\sup_{0 \le x_j \le T} \left| \frac{\lfloor kx_j\rfloor }{k} - \frac{n}{k} U_{(\lfloor kx_j\rfloor )}^j  \right| &~=~  \sup_{0 \le x_j \le T} \left| \frac{1}{k} \sum_{i=1}^n \mathds{1}_{\{ U_{i}^j \le U_{(\lfloor kx_j\rfloor )}^j   \}} - \frac{n}{k} U_{(\lfloor kx_j\rfloor )}^j  \right|\\
&~=~ \frac{n}{k} \sup_{0 \le x_j \le T} \left| \frac{1}{n} \sum_{i=1}^n \mathds{1}_{\{ U_{i}^j \le U_{(\lfloor kx_j\rfloor )}^j   \}} - \mathbb{P} \left [ U_1^j \le  U_{(\lfloor kx_j\rfloor )}^j \right] \right|\\
&~=~ \sup_{0 \le x_j \le T} \Theta_j (\frac{n}{k}U_{(\lfloor k x_j\rfloor )}^j ), 
\end{align*}
where $\Theta_j(y) = \frac{n}{k} 
\left| \frac{1}{n} \sum_{i=1}^n \mathds{1}_{\{ U_{i}^j \le \frac{k}{n}y   \}} - \mathbb{P} \left [ U_1^j \le \frac{k}{n} y \right] \right|$. 
Then, by (\ref{eq-Wellner}), 
with probability greater than $1-\delta$,
\begin{align*}
  \max_{1 \le j \le d} 
\sup_{0 \le x_j \le T} \left| \frac{\lfloor kx_j\rfloor }{k}
    - \frac{n}{k} U_{(\lfloor kx_j\rfloor )}^j \right|
~\le~
 \max_{1 \le j \le d} 
 \sup_{0 \le y \le 2T} \Theta_j(y)
\end{align*}
and from (\ref{Qialt2}), each term $\sup_{0 \le y \le 2T} \Theta_j(y)$
 is bounded by
$C\sqrt{\frac{T}{k}\log{\frac{1}{\delta}}}$ (with probability
$1-\delta$). In the end, 
 with probability greater than $1-(d+1) \delta$ :
\begin{align*}
\max_{1 \le j \le d} \sup_{0 \le y \le 2T} \Theta_j(y) ~\le~ C\sqrt{\frac{T}{k}\log{\frac{1}{\delta}}}~,
\end{align*}
which is  the desired inequality 
\end{proof}

 We may now proceed with the proof of Theorem \ref{thm:l}.
First of all, noticing that $\tilde F(t\mathbf{x})$ is non-decreasing in $x_j$ for every $l$ and that $l(\mathbf{x})$ is non-decreasing and continuous (thus uniformly continuous on $[0,T]^d$), from (\ref {stdf}) it is easy to prove by subdivising $[0,T]^d$ (see \cite{Qi97} p.174 for details) that 
\begin{align}
\label{unif_conv}
\sup_{0 \le \mathbf{x} \le T}\left| \frac{1}{t} \tilde F ( t \mathbf{x})-l(\mathbf{x}) \right|  \to 0 \text{~~~ as~~ t $\to$ 0 . }
\end{align}
\noindent
Using Lemma \ref{ln-Fn}, we can write :
\begin{align*}
\sup_{0 \le \mathbf{x} \le T} \left| l_n(\mathbf{x}) - l(\mathbf{x}) \right| &~=~ \sup_{0 \le \mathbf{x} \le T} \left| \frac{n}{k} \tilde F_n \left( U_{(\lfloor kx_1\rfloor )}^1,\ldots, U_{(\lfloor kx_d\rfloor )}^d \right) - l(\mathbf{x}) \right| \\
& ~\le~~~ \sup_{0 \le \mathbf{x} \le T} \left| \frac{n}{k} \tilde F_n \left(U_{(\lfloor kx_1\rfloor )}^1,\ldots, U_{(\lfloor kx_d\rfloor )}^d \right) - \frac{n}{k} \tilde F \left(U_{(\lfloor kx_1\rfloor )}^1,\ldots, U_{(\lfloor kx_d\rfloor )}^d \right)  \right| 
\\&~~~~~ + \sup_{0 \le \mathbf{x} \le T} \left| \frac{n}{k} \tilde F \left(U_{(\lfloor kx_1\rfloor )}^1,\ldots, U_{(\lfloor kx_d\rfloor )}^d \right) - l \left(\frac{n}{k} U_{(\lfloor kx_1\rfloor )}^1,\ldots, \frac{n}{k} U_{(\lfloor kx_d\rfloor )}^d \right) \right|
\\&~~~~~ + \sup_{0 \le \mathbf{x} \le T} \left| l \left(\frac{n}{k} U_{(\lfloor kx_1\rfloor )}^1, \ldots,\frac{n}{k} U_{(\lfloor kx_d\rfloor )}^d \right) - l(\mathbf{x}) \right|
\\&~=:~~~ \Lambda(n) ~~+~~ \Xi(n) ~~+~~ \Upsilon(n)~.
\end{align*}
\noindent
Now, by (\ref{eq-Wellner}) we have with probability greater than $1-\delta$ :
\begin{align*} 
\Lambda(n) ~\le~ \sup_{0 \le \mathbf{x} \le 2T}\left|\frac{n}{k} \tilde F_n(\frac{k}{n}\mathbf{x})-\frac{n}{k} \tilde F ( \frac{k}{n} \mathbf{x})\right|
\end{align*}
\noindent
and by Lemma \ref{Fn-tildeF}, 
\begin{align*}
 \Lambda(n) \le Cd \sqrt{\frac{2 T}{k}\log\frac{1}{\delta}}   
\end{align*}
\noindent
with probability at least $1-2\delta$. Similarly,
\begin{align*} 
\Xi(n) &~\le~  \sup_{0 \le \mathbf{x} \le  2 T}\left|\frac{n}{k}
  \tilde F(\frac{k}{n}\mathbf{x})-\frac{n}{k} l ( \frac{k}{n}
  \mathbf{x})\right| = 
 \sup_{0 \le \mathbf{x} \le  2 T} \left|\frac{n}{k}
  \tilde F(\frac{k}{n}\mathbf{x})- l (
  \mathbf{x})\right| ~\to~0 \quad\text{ (bias term)} 
\end{align*}
by virtue of (\ref{unif_conv}). Concerning $\Upsilon(n)$, we have :

\begin{align*}
 \Upsilon(n) &~\le~  \sup_{0 \le \mathbf{x} \le T} \left| l \left(\frac{n}{k} U_{(\lfloor kx_1\rfloor )}^1,\ldots, \frac{n}{k} U_{(\lfloor kx_d\rfloor )}^d \right) - l(\frac{\lfloor kx_1\rfloor }{k},\ldots,\frac{\lfloor kx_d\rfloor }{k}) \right| 
\\&~~~ ~+~  \sup_{0 \le \mathbf{x} \le T} \left| l(\frac{\lfloor kx_1\rfloor }{k},\ldots,\frac{\lfloor kx_d\rfloor }{k})-l(\mathbf{x}) \right| 
\\&~=~ \Upsilon_1(n) ~+~ \Upsilon_2(n)
\end{align*}
\noindent
Recall that $l$ is 1-Lipschitz on $[0,T]^d$ regarding to the $\|.\|_1$-norm, so that
\begin{align*}
\Upsilon_1(n) ~\le~ \sup_{0 \le \mathbf{x} \le T}  \sum_{l=1}^{d} \left| \frac{\lfloor kx_j\rfloor }{k} - \frac{n}{k} U_{(\lfloor kx_j\rfloor )}^j \right| 
\end{align*}
\noindent
so that by Lemma \ref{U-x}, with probability greater than $1-(d+1)\delta$:
\begin{align*}
\Upsilon_1(n) &~\le~  Cd \sqrt{\frac{2 T}{k}\log{\frac{1}{\delta}}}~.
\end{align*}
\noindent
On the other hand, $\Upsilon_2(n) ~\le~ \sup_{0 \le \mathbf{x} \le T} \sum_{l=1}^{d} \left|\frac{\lfloor k x_j\rfloor }{k} - x_j\right| ~\le~ \frac{d}{k}  $. 
Finally we get, for every $n >0$, with probability at least $1- (d+3)\delta$:
\begin{align*}
& \sup_{0 \le \mathbf{x} \le T} \left| l_n(\mathbf{x}) - l(\mathbf{x}) \right| ~\le~ \Lambda(n) + \Upsilon_1(n) + \Upsilon_2(n) + \Xi(n)
\\ &~~~~~~~~\le~  Cd\sqrt{\frac{2T}{k}\log\frac{1}{\delta}} ~+~ Cd\sqrt{\frac{2T}{k}\log\frac{1}{\delta}} ~+~ \frac{d}{k} ~+~\sup_{0 \le \mathbf{x} \le 2T}\left| \tilde F(\mathbf{x})-\frac{n}{k} l ( \frac{k}{n} \mathbf{x})\right|
\\ &~~~~~~~~\le~ C'd\sqrt{\frac{2T}{k}\log\frac{1}{\delta}} ~+~ \sup_{0 \le \mathbf{x} \le 2T}\left|\frac{n}{k} \tilde F(\frac{k}{n}\mathbf{x})- l ( \mathbf{x})\right|
\end{align*}

\section{Discussion}\label{sec:conclusion}

We provide a  non-asymptotic  bound  of VC type controlling 
the error 
of the 
empirical version of the \textsc{stdf}. 
Our bound
achieves the expected rate in $O(k^{-1/2}) + \text{bias}(k)$, where
$k$ is the number of (extreme) observations retained in the learning
process.  
In practice the smaller  $k/n$,  the smaller  the bias. Since no assumption is made on the underlying distribution, other than the existence of the
\textsc{stdf}, it is not possible in our framework to control the
bias explicitly. One option would be to make an additional hypothesis
of  `second order regular
variation'  \citep[see \emph{e.g.}][]{de1996second}. We made the
choice of making as 
few assumptions as possible, however, since the bias term is separated
from the `variance' term, it is probably feasible to refine our result
with more assumptions.

For the  purpose of controlling the empirical \textsc{stdf}, 
we have adopted the more general framework of maximal deviations in low
probability regions. The VC-type bounds adapted to  low probability regions derived in
Section~\ref{sec:concentration}    
may  directly be applied to a particular prediction context, namely where
the objective is to learn a classifier (or a regressor) that has good properties on
low probability regions.    
This may open the road to the study of classification of  extremal
observations, with immediate applications to the field of anomaly detection.





\bibliography{mvextrem}

\appendix

\section{Proof of Theorem \ref{thm-princ}}

Theorem~\ref{thm-princ} is actually  a short version of
Theorem~\ref{thm-princ-general} below:
\begin{theorem}[Maximal deviations]
\label{thm-princ-general}
Let $\mathbf{X}_1,\ldots,\mathbf{X}_n$ \iid~realizations of a \rv~$\mathbf{X}$ valued in $\mathbb{R}^d$, a VC-class $\mathcal{A}$, and denote by $\mathcal{R}_{n,p}$ the associated relative Rademacher average defined by 
\begin{align}
\label{def-extr-rad}
\mathcal{R}_{n,p} = \mathbb{E} \sup_{A \in \mathcal{A}} \frac{1}{np} \left | \sum_{i=1}^{n} \sigma_i \mathds{1}_{\mathbf{X}_i \in A}\right|~. 
\end{align}
Define the union $\mathbb{A} = \cup_{A \in \mathcal{A}} A$, and $ p = \mathbb{P}(\mathbf{X} \in \mathbb{A})$. Fix $0<\delta<1$, then with probability at least $1-\delta$,
\begin{align*}
\frac{1}{p} \sup_{A \in \mathcal{A}} \left | \mathbb{P}(\mathbf{X} \in A) - \frac{1}{n} \sum_{i=1}^n \mathds{1}_{\mathbf{X}_i \in A} \right| ~~\le~~ 2 \mathcal{R}_{n,p} ~+~ \frac{2}{3np} \log \frac{1}{\delta} ~+~ 2 \sqrt{\frac{1}{np} \log \frac{1}{\delta}}~,
\end{align*}
and there is a constant $C$ independent of $n,p,\delta$ such that with probability greater than $1- \delta$,
\begin{align*}
\sup_{A \in \mathcal{A}} \left | \mathbb{P}(\mathbf{X} \in A) - \frac{1}{n} \sum_{i=1}^n \mathds{1}_{\mathbf{X}_i \in A} \right| ~~\le~~ C \left( \sqrt{p} \sqrt{\frac{V_{\mathcal{A}}}{n} \log{\frac{1}{\delta}}} ~+~ \frac{1}{n} \log \frac{1}{\delta} \right)~.
\end{align*}
\noindent
If we assume furthermore that $\delta \ge e^{-np}$, then we both have:

\begin{align*}
&\frac{1}{p} \sup_{A \in \mathcal{A}} \left | \mathbb{P}(\mathbf{X} \in A) - \frac{1}{n} \sum_{i=1}^n \mathds{1}_{\mathbf{X}_i \in A} \right| ~~\le~~ 2 \mathcal{R}_{n,p}  ~+~ 3 \sqrt{\frac{1}{np} \log \frac{1}{\delta}}\\
&\frac{1}{p} \sup_{A \in \mathcal{A}} \left | \mathbb{P}(\mathbf{X} \in A) - \frac{1}{n} \sum_{i=1}^n \mathds{1}_{\mathbf{X}_i \in A} \right| ~~\le~~ C \sqrt{\frac{V_{\mathcal{A}}}{np} \log{\frac{1}{\delta}}}~.
\end{align*}
\end{theorem}

In the following, $\mb X_{1:n}$ denotes an \iid~sample $(\mb
X_1,\ldots,\mb X_n)$ distributed as  $\mb X$,  a $\bb R^d$-valued random vector. The classical steps to prove VC inequalities consist in applying a
concentration inequality to the function
\begin{equation}
  \label{eq:fsupdev}
  f(\mathbf{X}_{1:n}):= \sup_{A \in \mathcal{A}}
\left | \mathbb{P}(\mathbf{X} \in A) - \frac{1}{n} \sum_{i=1}^n
  \mathds{1}_{\mathbf{X}_i \in A} \right|, 
\end{equation}
 and then  establishing  bounds
on the expectation $\mathbb{E}f(\mathbf{X}_{1:n})$, using for instance
Rademacher average. Here we follow the same lines, but applying a
Bernstein type concentration inequality instead of the usual Hoeffding
one, since the variance term in the bound involves the probability $p$
to be in the union of the VC-class $\mathcal{A}$ considered. We then
introduce relative Rademacher averages instead of the conventional
ones, to take into account  $p$ for bounding  $\mathbb{E}f(\mathbf{X}_{1:n})$. 

We need first to  control the variability of the random variable $f(\mathbf{X}_{1:n})$ when fixing all but one marginal $\mathbf{X}_i$. For that purpose introduce the functional
\begin{align*}
h(\mathbf{x}_1,\ldots,\mathbf{x}_k) = \mathbb{E}\left [ f(\mb X_{1:n}) | \mathbf{X}_1=\mathbf{x}_1,\ldots,\mathbf{X}_k=\mathbf{x}_k \right] - \mathbb{E}\left [ f(\mb X_{1:n}) | \mathbf{X}_1=\mathbf{x}_1,\ldots,\mathbf{X}_{k-1}=\mathbf{x}_{k-1} \right]
\end{align*}
The \emph{positive deviation} of
$h(\mathbf{x}_1,\ldots,\mathbf{x}_{k-1},\mathbf{X}_k)$ is defined by 
\[dev^+(\mathbf{x}_1,\ldots,\mathbf{x}_{k-1})= \sup_{\mathbf{x} \in
  \mathbb{R}^d} \left \{
  h(\mathbf{x}_1,\ldots,\mathbf{x}_{k-1},\mathbf{x})\right \}, \]
 and $\text{maxdev}^+$, the maximum of all positive deviations, by 
 \[\text{maxdev}^+ = \sup_{\mathbf{x}_1,\ldots,\mathbf{x}_{k-1}} \max_{k}\;
 dev^+(\mathbf{x}_1,\ldots,\mathbf{x}_{k-1})~.\] 
 Finally,  define $\hat v $, the \emph{maximum sum of
 variances}, by $$\hat v = \sup_{\mathbf{x}_1,\ldots,\mathbf{x}_n}
 \sum_{k=1}^{n} \Var~ h(\mathbf{x}_1,\ldots,\mathbf{x}_{k-1}, \mb X_k)~.$$ We have
 now the tools to state an extension of the classical Bernstein
 inequality, which is proved in \cite{McDiarmid98}.
\begin{proposition}
\label{thm-berstein}
Let $\mathbf{X}_{1:n} = (\mathbf{X}_1,\ldots,\mathbf{X}_n)$ 
as above, and $f$ any function $(\bb R^d)^n\to \bb R$~.  Let $\text{maxdev}^+$
and $\hat v$ the maximum sum of variances, both of which we assume to
be finite, and let $\mu$ be the mean of $f(\mb X_{1:n})$. Then for any $t \ge 0$,
\begin{align*}
\mathbb{P} \big[ f(\mb X_{1:n}) - \mu \ge t \big] ~\le~ \exp{\left(- \frac{t^2}{2 \hat v (1 + \frac{\text{maxdev}^+ t}{3 \hat v}) }\right)}~.
\end{align*}
\end{proposition}
\noindent
Note that the term $\frac{\text{maxdev}^+ t}{3 \hat v}$ is view as an `error
term' and is often negligible. Let us apply this theorem to the
specific  function $f$ defined in (\ref{eq:fsupdev}).  
Then the following lemma holds true:

\begin{lemma}
\label{lem-1}
In the situation of Proposition \ref{thm-berstein} with $f$ as in
(\ref{eq:fsupdev}),  we have
$$\text{maxdev}^+ \le \frac{1}{n} \text{~~and~~} \hat v \le \frac{q}{n}, $$ where 
\begin{align}
\label{lem-1:q}
q ~=~ \mathbb{E}\left ( \sup_{A \in \mathcal{A}} \left |
    \mathds{1}_{\mathbf{X}' \in A} - \mathds{1}_{\mathbf{X} \in A}
  \right|\right) ~\le~ 2 \mathbb{E}\left ( \sup_{A \in \mathcal{A}}
  \left | \mathds{1}_{\mathbf{X}' \in A}  \mathds{1}_{\mathbf{X}
      \notin A} \right|\right), 
\end{align}
with $\mathbf{X}'$ an independent copy of $\mathbf{X}$.

\end{lemma}

\noindent
\begin{proof}
Considering the definition of $f$, we have:
\begin{align*}
h(\mathbf{x}_1,\ldots,\mathbf{x}_{k-1},\mathbf{x}_k) &= ~\mathbb{E} \sup_{A \in \mathcal{A} } \left | \mathbb{P}(\mathbf{X} \in A) - \frac{1}{n} \sum_{i=1}^k \mathds{1}_{\mathbf{x}_i \in A} - \frac{1}{n} \sum_{i=k+1}^n \mathds{1}_{\mathbf{X}_i \in A}  \right|  \\ 
&~~~~~~~~~~~~~~~~~~~~~~~~~~~~~~~~~~-~ \mathbb{E} \sup_{A \in \mathcal{A} } \left | \mathbb{P}(\mathbf{X} \in A) - \frac{1}{n} \sum_{i=1}^{k-1} \mathds{1}_{\mathbf{x}_i \in A} - \frac{1}{n} \sum_{i=k}^n \mathds{1}_{\mathbf{X}_i \in A}  \right|~. 
\end{align*}
Using the fact that $\big | \sup_{A \in \mathcal{A}}|F(A)| - \sup_{A \in \mathcal{A}}|G(A)| \big| \le \sup_{A \in \mathcal{A}} |F(A) - G(A)|$ for every function $F$ and $G$ of $A$, we obtain:
\begin{align}
\label{g1}
\big| h(\mathbf{x}_1,\ldots,\mathbf{x}_{k-1},\mathbf{x}_k) \big| ~\le~ \mathbb{E}  \sup_{A \in \mathcal{A}} \frac{1}{n} \left | \mathds{1}_{\mathbf{x}_k \in A} - \mathds{1}_{\mathbf{X}_k \in A}  \right|~.
\end{align}
\noindent
The term on the right hand side of (\ref{g1}) is less than $\frac{1}{n}$ so that $\text{maxdev}^+ \le \frac{1}{n}$. Moreover, if $\mathbf{X}'$ is an independent copy of $\mathbf{X}$, (\ref{g1}) yields  
\begin{align*}
\big| h(\mathbf{x}_1,\ldots,\mathbf{x}_{k-1},\mathbf{X}') \big| ~\le~ \mathbb{E} \left [  \sup_{A \in \mathcal{A}} \frac{1}{n} \left | \mathds{1}_{\mathbf{X}' \in A} - \mathds{1}_{\mathbf{X} \in A}  \right| ~\Big|~ \mathbf{X}' \right],
\end{align*}

\noindent
so that 
\begin{align*}
\mathbb{E} \left [ h(\mathbf{x}_1,\ldots,\mathbf{x}_{k-1},\mathbf{X}')^2\right] &~\le~ \mathbb{E}~ \mathbb{E} \left [  \sup_{A \in \mathcal{A}} \frac{1}{n} \left | \mathds{1}_{\mathbf{X}' \in A} - \mathds{1}_{\mathbf{X} \in A}  \right| ~\Big|~ \mathbf{X}' \right]^2 \\
&~\le~ \mathbb{E} \left [  \sup_{A \in \mathcal{A}} \frac{1}{n^2} \left | \mathds{1}_{\mathbf{X}' \in A} - \mathds{1}_{\mathbf{X} \in A}  \right|^2 \right] \\
&~\le~ \frac{1}{n^2} \mathbb{E} \left [  \sup_{A \in \mathcal{A}}  \left | \mathds{1}_{\mathbf{X}' \in A} - \mathds{1}_{\mathbf{X} \in A}  \right|  \right] 
\end{align*}
\noindent
Thus $\Var(h(\mathbf{x}_1,\ldots,\mathbf{x}_{k-1},\mathbf{X}_k)) \le \mathbb{E} [h(\mathbf{x}_1,\ldots,\mathbf{x}_{k-1},\mathbf{X}_k)^2]$ $ \le \frac{q}{n^2}$. Finally $\hat v \le \frac{q}{n}$ as required.
\end{proof}


\noindent
As a consequence with Proposition \ref{thm-berstein} the following general inequality holds true:
\begin{align}
\label{concentration-general}
\mathbb{P}\left [ f(\mb X_{1:n}) - \mathbb{E} f(\mb X_{1:n}) ~\ge~ t \right] ~\le~ e^{-\frac{n t^2}{2q + \frac{2t}{3}} }
\end{align}
\noindent
where the quantity $q~=~ \mathbb{E}\left ( \sup_{A \in \mathcal{A}} \left | \mathds{1}_{\mathbf{X}' \in A} - \mathds{1}_{\mathbf{X} \in A} \right|\right)$ 
seems to be a central characteristic of the VC-class $\mathcal{A}$ given the distribution $\mathbf{X}$. It may be interpreted as a measure of the complexity of the class $\mathcal{A}$ with respect to the distribution of $\mathbf{X}$: how often the class $\mathcal{A}$ is able to separate two independent realizations of $\mathbf{X}$. 

Recall that the union class $\mathbb{A}$ and its associated probability $p$ are defined as $\mathbb{A} = \cup_{A \in \mathcal{A}} A$, and $p = \mathbb{P}(\mathbf{X} \in \mathbb{A})$. Noting that for all $A \in \mathcal{A}$, $\mathds{1}_{\{ . \in A\}} \le \mathds{1}_{\{ . \in \mathbb{A}\}}$, it is then straightforward from (\ref{lem-1:q}) that $q \le 2p$. As a consequence (\ref{concentration-general})  holds true when changing $q$ by $2p$.
Let us now explicit the link between the expectation of $f$ and the Rademacher average $$\mathcal{R}_n = \mathbb{E} \sup_{A \in \mathcal{A}} \frac{1}{n} \left | \sum_{i=1}^{n} \sigma_i \mathds{1}_{\mathbf{X}_i \in A}\right|~,$$ where $(\sigma_i)_{i \ge 1}$ is a Rademacher chaos independent of the $\mathbf{X}_i$'s.
\begin{lemma}
\label{lem-rademacher}
 With this notations the following inequality holds true:
$$ \mathbb{E}f(\mb X_{1:n}) ~\le~ 2 \mathcal{R}_n$$
\end{lemma}

\begin{proof}
The proof of this lemma relies on classical arguments: Introducing a ghost sample $(\mathbf{X}_i^{'})_{1 \le i \le n}$ namely i.i.d independent copy of the $\mathbf{X}_i$'s, we may write:
\begin{align*}
\mathbb{E} f(\mb X_{1:n}) &~=~ \mathbb{E} \sup_{A \in \mathcal{A}} \left | \mathbb{P}(\mathbf{X} \in A) - \frac{1}{n} \sum_{i=1}^{n}\mathds{1}_{\mathbf{X}_i \in A} \right| \\
&~=~ \mathbb{E} \sup_{A \in \mathcal{A}} \left | \mathbb{E}\left [ \frac{1}{n} \sum_{i=1}^{n}\mathds{1}_{\mathbf{X}_i^{'} \in A}\right] - \frac{1}{n} \sum_{i=1}^{n}\mathds{1}_{\mathbf{X}_i \in A} \right|\\
&~\le~ \mathbb{E} \sup_{A \in \mathcal{A}} \left | \frac{1}{n} \sum_{i=1}^{n}\mathds{1}_{\mathbf{X}_i^{'} \in A} - \frac{1}{n} \sum_{i=1}^{n}\mathds{1}_{\mathbf{X}_i \in A} \right| \\
&~=~ \mathbb{E} \sup_{A \in \mathcal{A}} \left |\frac{1}{n} \sum_{i=1}^{n} \sigma_i \left ( \mathds{1}_{\mathbf{X}_i^{'} \in A} - \mathds{1}_{\mathbf{X}_i \in A} \right) \right|\\
&~\le~\mathbb{E} \sup_{A \in \mathcal{A}} \left |\frac{1}{n} \sum_{i=1}^{n} \sigma_i \mathds{1}_{\mathbf{X}_i^{'} \in A}  \right| ~+ \sup_{A \in \mathcal{A}} \left |\frac{1}{n} \sum_{i=1}^{n} -\sigma_i \mathds{1}_{\mathbf{X}_i \in A}  \right|\\
&~=~ 2 \mathcal{R}_n
\end{align*}
\end{proof}
\noindent
Combining (\ref{concentration-general}) with Lemma \ref{lem-rademacher} and the fact that $q \le 2p$ gives:
\begin{align}
\label{vc1}
\mathbb{P}\left [ f(\mb X_{1:n}) - 2 \mathcal{R}_n \ge t \right] ~\le~ e^{-\frac{n t^2}{4p + \frac{2t}{3}} }~.
\end{align}


Recall that the relative Rademacher average are defined in (\ref{def-extr-rad}) as $\mathcal{R}_{n,p} = \mathcal{R}_n /p$. 
It is well-known that $\mathcal{R}_n$ is of order $\mathcal{O}( (V_{\mathcal{A}}/n)^{1/2})$, see \cite{Kolt06} for instance. However, we hope a stronger bound than just $\mathcal{R}_{n,p} = \mathcal{O}(p^{-1}(V_{\mathcal{A}}/n)^{1/2})$ since $\frac{1}{np} \left | \sum_{i=1}^{n} \sigma_i \mathds{1}_{\mathbf{X}_i \in A}\right|$ with $\mathbb{P}(\mathbf{X}_i \in \mathbb{A})=p$ is expected to be like $\frac{1}{np} \left | \sum_{i=1}^{np} \sigma_i \mathds{1}_{\mathbf{Y}_i \in A}\right|$ with $\mathbf{Y}_i$ such that $\mathbb{P}(\mathbf{Y}_i \in \mathbb{A}) = 1$. The result below confirms this heuristic:

\begin{lemma}
\label{lem-relative-rademacher}
The relative Rademacher average ${\mathcal{R}}_{n,p}$ is of order $\mathcal{O}(\sqrt\frac{V_{\mathcal{A}}}{pn})$. 
\end{lemma}

\begin{proof}
Let us defined \iid~\rv~$\mathbf{Y}_i$ independent from $\mathbf{X}_i$ whose law is the law of $\mathbf{X}$ conditioned on the event $\mathbf{X} \in \mathbb{A}$. If $\overset{d}{=}$ means equal in distribution it is easy to show that
$\sum_{i=1}^n \sigma_i \mathds{1}_{\mathbf{X}_i \in A} \overset{d}{=} \sum_{i=1}^{\kappa} \sigma_i \mathds{1}_{\mathbf{Y}_i \in A}$, where $\kappa \sim Bin(n,p)$ independent of the $\mathbf{Y}_i$'s. 
Thus, 
\begin{align*}
\mathcal{R}_{n,p} ~=~ \mathbb{E} \sup_{A \in \mathcal{A}} \frac{1}{np} \left | \sum_{i=1}^{n} \sigma_i \mathds{1}_{\mathbf{X}_i \in A}\right| &~=~ \mathbb{E} \sup_{A \in \mathcal{A}} \frac{1}{np} \left | \sum_{i=1}^{\kappa} \sigma_i \mathds{1}_{\mathbf{Y}_i \in A}\right| \\
&~=~ \mathbb{E} \left[\mathbb{E} \left[\sup_{A \in \mathcal{A}} \frac{1}{np} \left | \sum_{i=1}^{\kappa} \sigma_i \mathds{1}_{\mathbf{Y}_i \in A} \right|~~|~\kappa \right] \right]\\
&~=~\mathbb{E}\left [ \Phi(\kappa)\right]
\end{align*}
where $$\phi(K) = \mathbb{E} \left[\sup_{A \in \mathcal{A}} \frac{1}{np} \left | \sum_{i=1}^{K} \sigma_i \mathds{1}_{\mathbf{Y}_i \in A} \right|\right] = \frac{K}{np} \mathcal{R}_K \le \frac{K}{np} \frac{C \sqrt{V_{\mathcal{A}}}}{\sqrt K}~.$$
Thus,
\begin{align*}
\mathcal{R}_{n,p} ~\le~ \mathbb{E}\left [ \frac{\sqrt \kappa}{np} C\sqrt{V_{\mathcal{A}}}\right] ~\le~ \frac{\sqrt{\mathbb{E}[\kappa]}}{np} C\sqrt{V_{\mathcal{A}}} ~\le~ \frac{C\sqrt{V_{\mathcal{A}}}}{\sqrt{np}} ~.
\end{align*}

\end{proof}


\noindent
Finally we obtain from (\ref{vc1}) and Lemma \ref{lem-relative-rademacher} the following bound:
\begin{align}
\label{prop-princ}
\mathbb{P} \left [ \frac{1}{p} \sup_{A \in \mathcal{A}} \left | \mathbb{P}(\mathbf{X} \in A) - \frac{1}{n} \sum_{i=1}^n \mathds{1}_{\mathbf{X}_i \in A} \right| - 2 \mathcal{R}_{n,p} ~>~ t \right] ~~\le~~ e^{-\frac{np t^2}{4 + \frac{2t}{3}} }
\end{align}
\noindent
Solving $ \exp \left [ - \frac{np t^2}{4  + \frac{2}{3} t}\right]=\delta$ with $t > 0$ leads to $$t ~=~ \frac{1}{3np} \log \frac{1}{\delta} + \sqrt{\left ( \frac{1}{3np} \log \frac{1}{\delta}\right)^{2} + \frac{4}{np} \log \frac{1}{\delta} } ~:=~h(\delta)$$ so that
\begin{align*}
\mathbb{P} \left [ \frac{1}{p} \sup_{A \in \mathcal{A}} \left | \mathbb{P}(\mathbf{X} \in A) - \frac{1}{n} \sum_{i=1}^n \mathds{1}_{\mathbf{X}_i \in A} \right| - 2 \mathcal{R}_{n,p} ~> h(\delta) \right] ~~\le~~ \delta
\end{align*}
\noindent
Using $\sqrt{a+b} \le \sqrt a + \sqrt b~$ if $a,b \ge 0$, we have $h(\delta) < \frac{2}{3np} \log \frac{1}{\delta} + 2 \sqrt{\frac{1}{np} \log \frac{1}{\delta}} $. In the case of $\delta \ge e^{-np}$, $\frac{2}{3np} \log \frac{1}{\delta} \le \frac{2}{3} \sqrt{ \frac{1}{np}\log \frac{1}{\delta}} $ so that 
$h(\delta) < 3 \sqrt{\frac{1}{np} \log \frac{1}{\delta}} $. This ends the proof.

\section{Note on Remark \ref{rk:prediction}}
To obtain the bound in (\ref{prediction:rates}), the following easy to show inequality is needed before  applying Theorem~\ref{thm-princ} :
\begin{align*}
&\sup_{g\in \mathcal{G}}| {L}_{\alpha, n}(g)- L_{\alpha}(g)  | ~~ \le ~~ \frac{1}{\alpha} \Bigg[ \sup_{g\in \mathcal{G}} \left| \mathbb{P} \left(Y\neq g(\mathbf{X}),~ \| \mathbf{X}\|\ >t_\alpha \right) - \frac{1}{n}\sum_{i=1}^n\mathds{I}_{\{Y_i\neq
   g(\mathbf{X}_i),~ \| \mathbf{X}_i\| > t_\alpha \}} \right| \\
&~~~~~~~~~~~~~~~~~~~~~~~~~~~~~~~~~~~~~~~~~~~~~~~~~~~~~~~~~~~~~~~~~~~~~~~~+~ \left| \mathbb{P} \left(\| \mathbf{X}\|\ >t_\alpha \right) - \frac{1}{n}\sum_{i=1}^n\mathds{I}_{\{ \| \mathbf{X}_i\| > t_\alpha  \}} \right| ~~+~~ \frac{1}{n}  \Bigg] ~.
\end{align*}

Note that the final objective would be   to bound the quantity 
$\sup_{g\in \mathcal{G}}| {L}_{\alpha}(g)- L_{\alpha}(g^*_\alpha)  |$,
where $g^*_\alpha$ is a Bayes classifier for the problem at stake,
\ie~a  solution of the conditional risk  minimization problem  $\inf_{\{ g \text{
    measurable}\}} L_\alpha(g)$. Such a bound involves a  bias term $ \inf_{g\in
    \mathcal{G}}L_\alpha(g)-L_\alpha(g_\alpha^*)$, as in the classical
  setting. Further, it  can  be shown that the standard Bayes classifier
  $g^*(\mb x) := 2\mathbb{I}\{\eta(\mathbf{x})>1/2\}-1$ (where $\eta(\mb
  x) =
  \P(Y=1\; |\; \mb X=\mb x)$) is also   a solution of the conditional
  risk minimization problem. 
 Finally,  the conditional bias
 $ \inf_{g\in \mathcal{G}}L_\alpha(g)-L_\alpha(g_\alpha^*)$ can
 be expressed as
$\frac{1}{\alpha}\inf_{g \in \mathcal{G}} \mathbb{E} \left [ |2
  \eta(\mathbf{X})-1|\mathds{1}_{g(\mathbf{X}) \neq g^*(\mathbf{X})
  }\mathds{1}_{ \|\mathbf{X}\| \ge t_\alpha}\right] $, to be compared
with the standard bias $\inf_{g \in \mathcal{G}} \mathbb{E} \left [ |2
  \eta(\mathbf{X})-1|\mathds{1}_{g(\mathbf{X}) \neq g^*(\mathbf{X})
  }\right]$.




\end{document}